\documentclass{gtpart}     
\usepackage{color,graphicx}

\title{Upper bounds for the complexity of torus knot complements}

\author{Evgeny Fominykh}
\givenname{Evgeny}
\surname{Fominykh}
\address{Evgeny Fominykh, Department of Mathematics, Chelyabinsk State University, Brat'ev Kashirinykh Street 129, Chelyabinsk 454001,  and Institute of Mathematics and Mechanics, Ural Branch of the Russian Academy of Sciences, S. Kovalevskaja street 16, Ekaterinburg, 620990, Russia}
\email{fominykh@csu.ru}
\urladdr{http://topology.math.csu.ru/stuffs/eng/fominykh/index.htm}

\author{Bert Wiest}
\givenname{Bert}
\surname{Wiest}
\address{Bert Wiest, UFR Math\'ematiques, Universit\'e de Rennes 1, 35042 Rennes Cedex, France}
\email{bertold.wiest@univ-rennes1.fr}
\urladdr{http://perso.univ-rennes1.fr/bertold.wiest}


\volumenumber{}
\issuenumber{}
\publicationyear{}
\papernumber{}
\startpage{}
\endpage{}
\doi{}
\MR{}
\Zbl{}
\received{}
\revised{}
\accepted{}
\published{}
\publishedonline{}
\proposed{}
\seconded{}
\corresponding{}
\editor{}
\version{}

\makeautorefname{notation}{Notation}%

\newtheorem{theorem}{Theorem}[section]
\newtheorem{lemma}[theorem]{Lemma}
\newtheorem{proposition}[theorem]{Proposition}
\newtheorem{corollary}[theorem]{Corollary}

\theoremstyle{definition}
\newtheorem{definition}[theorem]{Definition}
\newtheorem{remark}[theorem]{Remark}

\newtheorem{example}[theorem]{Example}

\newtheorem{notation}[theorem]{Notation}

\def\C{{\mathbb C}}
\def\N{{\mathbb N}}

\def\R{{\mathbb R}}
\def\co{\colon \thinspace}
\def\marg#1{{\marginpar{\tiny #1}}}

\renewcommand{\phi}{\varphi}

\begin{document}

\begin{abstract}
We establish upper bounds for the complexity of Seifert fibered manifolds with nonempty boundary. In particular, we obtain potentially sharp bounds on the complexity of torus knot complements.
\end{abstract}

\maketitle


\section{Introduction and statement of the results}\label{S:Intro}

 Let $M$ be a compact $3$-manifold with nonempty boundary. Recall
\cite{Mat03} that a subpolyhedron $P\subset M$ is said to be a
\emph{spine} of $M$ if the manifold $M\setminus P$ is homeomorphic
to $\partial M \times [0,1)$. A spine $P$ is said to be \emph{almost
simple} if the link of each of its points can be embedded into a
complete graph $K_4$ with four vertices. A \emph{true vertex} of an
almost simple spine $P$ is a point whose link is~$K_4$. The
\emph{complexity} $c(M)$ of $M$ is defined as the minimum possible
number of true vertices of an almost simple spine of $M$.

The problem of calculating the complexity $c(M)$ of any given 3-manifold~$M$ is very difficult.
Exact values of the complexity are presently known only for certain
infinite series of irreducible boundary irreducible $3$-manifolds
\cite{FrigMarPet03-1, FrigMarPet03-2, Anis05, JacoRubTil09,
JacoRubTil11, VesFom11, VesFom12}. In addition, this problem is solved for all
closed orientable irreducible manifolds up to complexity~$12$
\cite{Atlas}.

By contrast, the task of giving an upper bound for the complexity is very easy. Indeed, for any 3-manifold $M$ which is given by practically any representation (e.g.\ by a triangulation, a surgery description, or a Heegaard diagram), it is known \cite{Mat03} how to construct an almost simple spine $P$ of $M$, and the number of true vertices of $P$ will serve as an upper bound for the complexity. However, this bound is usually not at all sharp. For instance, Matveev in~\cite{Mat03} gave an upper bound for the complexity of the complement of any link~$L$ in terms of simple combinatorial data in a diagram of~$L$. While his bound has the great advantage of providing a simple formula valid for all knots and links, the drawback is that it is not sharp, even for the trefoil knot (see Section~\ref{S:TorusKnots} below).

An intermediate task, of intermediate difficulty, is to construct upper bounds for the complexities of families of 3-manifolds which have some hope of being sharp, in the sense that they cannot be improved upon by any known
construction, and that they are sharp for simple examples, where the exact complexity is known (e.g. from~\cite{Atlas}). For instance, an important result in this direction has been obtained by B.~Martelli and C.~Petronio~\cite{MarPet04}: they found a potentially sharp upper bound for the complexity of all closed orientable Seifert manifolds.

The aim of the present paper is to give potentially sharp upper bounds
for the complexity of orientable Seifert fibered manifolds with nonempty boundary.

For the reader's convenience, we recall the definition of the Seifert manifold
$M=\big(F, (p_1,q_1),\ldots ,(p_k,q_k)\big)$,
$k\geqslant 0$, where $F$ is a compact surface with nonempty boundary
and $(p_i, q_i)$ are pairs of coprime integers with $|p_i|\geqslant 2$. Consider a surface $F_0$ obtained from $F$ by removing the interiors of $k$ disjoint disks. The boundary circles
of these disks are denoted by $c_1, \ldots, c_k$. Let $c_{k+1},
\ldots, c_n$ be all the remaining circles of~$\partial F$. 
Consider an orientable $S^1$-bundle $M_0$ over $F_0$. In other words, $M_0 = F_0\times S^1$ or $M_0 = F_0\widetilde{\times} S^1$, depending on whether or not $F_0$ is orientable.
We choose an orientation of $M_0$ 
and a section $s\co F_0 \to M_0$ of the projection map $p\co M_0\to F_0$. Each torus $T_i = p^{-1}(c_i)$, for $1\leqslant i\leqslant k$, carries an orientation, induced from the orientation of~$M_0$. On each~$T_i$ we choose a coordinate system, 
taking $s(c_i)$ as the meridian~$\mu_i$ and a fiber $p^{-1}(\{*\})$ as the longitude~$\lambda_i$. The orientations of the coordinate curves must satisfy the following conditions:
\begin{enumerate}
\item In case $M_0 = F_0\times S^1$ the orientations of $\lambda_i$ must be induced by a fixed orientation of $S^1$. If $M_0 = F_0\widetilde{\times} S^1$, then the orientations of $\lambda_i$ can be chosen arbitrarily.
\item The intersection number of $\mu_i$ with $\lambda_i$ in~$T_i$ must be $1$.
\end{enumerate}
Now, let us attach solid tori $V_i = D^2_i\times S^1$, $1\leqslant
i\leqslant k$, to $M_0$ via homeomorphisms $h_i = \partial V_i\to
T_i$ such that each $h_i$ takes the meridian $\partial D^2_i\times
\{*\}$ of $V_i$ into a curve $\mu^{p_i}\lambda^{q_i}$. The resulting
manifold is $M$. 

The Seifert parameters are called \emph{normalized} if $p_i > q_i > 0$
for all i. Since $\partial M\neq \emptyset$, any set of Seifert
parameters can be promoted to a normalized one by replacing each $(p_i, q_i)$ by $(|p_i|, q'_i)$ where $|p_i| > q'_i > 0$, $q'_i\equiv q_i\frac{p_i}{|p_i|} \ (\mathrm{mod\ } |p_i|)$.

\begin{notation} 
For two coprime integers $p,q$ with $0<q<p$, we denote by $S(p, q)$ the sum of all partial quotients in the
expansion of $p/q$ as a regular continued fraction, i.e. 
$$
\text{if \ \ \ }
p/q = a_1 + \frac{\displaystyle 1}{\displaystyle a_2 + \, \cdots \,
+ \frac{\displaystyle 1}{\displaystyle a_{r-1} + \frac{\displaystyle
1}{\displaystyle a_{r}}}}, 
\text{\ \ \ \ \ then \ \ \ } S(p, q) = a_1 + \ldots + a_{r}.
$$
\end{notation}

The main result of this paper is as follows.

\begin{theorem}\label{T:ComplSFS} 
Let $M=\big(F, (p_1,q_1),\ldots ,(p_k,q_k)\big)$ be an orientable Seifert fibered manifold with nonempty boundary and with normalized parameters. Then
 $$c(M)\leqslant \sum_{i=1}^k \max\{S(p_i, q_i)-3, 0\}.$$
\end{theorem}

If $k=0$, i.e.\ if $M$ has no singular fibres, then Theorem~\ref{T:ComplSFS} should be interpreted as saying that $c(M)=0$.

As an application of Theorem \ref{T:ComplSFS}, we obtain the following
upper bounds on the complexity of torus knots which we conjecture to
be sharp.

\begin{theorem}\label{T:ComplTorusKnots}
Suppose $\alpha$ and $\beta$ are two coprime integers with
$\alpha>\beta\geqslant 2$. Then the complexity $C$ of the complement of the
$(\alpha,\beta)$-torus knot satisfies
$$
\textstyle C\leqslant\max\left\{S(\alpha,\beta)-3\ , \
0\right\}+\max\left\{S(\alpha,\beta)-\lfloor\frac{\alpha}{\beta}\rfloor-3
\ , \ 0\right\}
$$
\end{theorem}

 As we shall see in Corollary~\ref{C:SharpForSimplest}, this bound is
sharp for the torus knots with parameters $(3,2)$, $(5,2)$, $(4,3)$
and $(5,3)$.

\begin{remark} It is more usual to talk about the $(p,q)$-torus knot.
However, we use the
notation $(\alpha,\beta)$-torus knot in order to avoid confusion
with the surgery parameters $(p_i,q_i)$.
\end{remark}

{\bf Acknowledgements } This research was carried out while the first author was visiting Rennes, financed by the CNRS (PICS grant number 5512) and RFBR (research project No.\ 10-01-91056). The first author was supported also by the grant for Russian Leading Scientific Schools (NSh1414.2012.1).


\section{Seifert fibered manifolds with boundary}\label{S:SFS}

The aim of this section is to prove Theorem~\ref{T:ComplSFS}. The proof will be decomposed into several subsections.

\subsection{Seifert fibered manifolds without singular fibers}

\begin{proposition}\label{P:NoSingularFibre} 
Let $M$ be an orientable $S^1$-bundle over 
a compact connected surface $F$ with nonempty boundary. Then the complexity of~$M$ is zero.
\end{proposition}
\begin{proof}
First we prove that there exists a graph $\Gamma\subset F$ such that 
each vertex of $\Gamma$ has valence at most $3$ and $F-\Gamma\cong \partial F\times [0,1)$. Let $b$ denote the number of boundary components of $F$. The following cases are possible:
\begin{enumerate}
 \item $\chi(F)< 0$. By gluing a disk to each boundary component of $F$ we
obtain a closed surface, which we call $G$. This surface has a singular triangulation with $b$ vertices, one in each glued-on disk. Denote by $\Gamma$ the dual graph of the triangulation of $G$. Since $F$ can be obtained from $G$ by removing small open neighborhoods of all vertices of the triangulation, we have $\Gamma\subset F$ and $F-\Gamma\cong \partial F\times [0,1)$.
 \item $\chi(F)=0$ and $b=2$. Then $F$ is an annulus and we choose as $\Gamma$ the core circle of $F$.
 \item $\chi(F)=0$ and $b=1$. Then $F$ is a M\"obius strip and we choose as $\Gamma$ the core circle of~$F$.
 \item $\chi(F)=1$ and $b=1$. Then $F$ is a disk and we choose as $\Gamma$ an interior point of~$F$.
\end{enumerate}

Let $p\co M\to F$ be the projection map. Then $p^{-1}(\Gamma)$ is an almost simple spine of~$M$. Indeed, $M-p^{-1}(\Gamma)\cong \partial M\times [0,1)$ and the link of each point of $p^{-1}(\Gamma)$ is homeomorphic to either a circle, or a circle with diameter, or two points. We see that in the spine $p^{-1}(\Gamma)$ there are no true vertices, so the complexity of~$M$ is zero.
\end{proof}


\subsection{Collapsing}

In order to discuss spines, we need to define the notion of collapsing.
Let~$K$ be a simplicial complex, and let $\triangle^n, \delta^{n-1}\in
K$ be two open simplices such that $\triangle^n$ is \emph{principal},
i.e., $\triangle^n$ is not a proper face of any simplex in $K$, and
$\delta^{n-1}$ is a \emph{free} face of it, i.e., $\delta^{n-1}$ is
not a proper face of any simplex in $K$ other than $\triangle^n$. The
transition from $K$ to $K\setminus (\triangle^n\cup \delta^{n-1})$ is
called an \emph{elementary simplicial collapse}. A polyhedron $P$
\emph{collapses} to a subpolyhedron $Q$ (notation: $P \searrow Q$)
if for some triangulation $(K,L)$ of the pair $(P,Q)$ the complex
$K$ collapses onto $L$ by a sequence of elementary simplicial
collapses. By a \emph{simplicial collapse} of a simplicial complex~$K$ onto
its subcomplex~$L$ we mean any sequence of elementary simplicial
collapses transforming~$K$ into~$L$. In general, there is no need to
triangulate~$P$ to construct a collapse $P \searrow Q$; for this
purpose one can use cells instead of simplexes.

As follows from \cite[Theorem 1.1.7]{Mat03}, the condition that a 
subpolyhedron
$P\subset M$ is a spine of a compact $3$-manifold~$M$ with boundary
is equivalent to $M \searrow P$.


\subsection{Skeleta}

Any almost simple polyhedron $P$ can be presented as the union of its 2-dimensional and its 1-dimensional parts. The 1-dimensional part (the closure of the set of points with
0-dimensional links) is a graph, the 2-dimensional part consists of points whose
links contain an arc. 

\begin{definition}
An almost simple polyhedron $P$ is called \emph{simple} if the link of
each point $x\in P$ is homeomorphic to either a circle, or to a circle with a diameter, or to $K_4$. 
\end{definition}

If an almost simple polyhedron cannot be collapsed onto a smaller subpolyhedron, then its 2-dimensional part is a simple polyhedron (maybe disconnected). 
In practice it is usually easier to construct an almost simple spine of a 3-manifold without 1-dimensional parts, i.e. a simple one. 

\begin{definition}
A theta-curve~$\theta$ on a torus~$T$ is a subset of~$T$ which is homeomorphic to the theta-graph (i.e.\ the graph with two vertices and three edges connecting them) such that $T\setminus \theta$ is an open disk.
\end{definition}

\begin{remark}
As with simple closed curves, we shall often talk about theta-curves when we actually mean isotopy classes of theta-curves.
Note that any theta-curve contains exactly three simple closed curves, each
obtained by removing one of the three edges.
\end{remark}

\begin{notation}
We denote by $\mathcal{T}$ the class of all compact 3-manifolds $M$ with nonemty boundary~$\partial M$ all of whose components are tori. Moreover, we 
shall always present~$\partial M$ as the disjoint union of two collections of boundary components $\partial_{-} M$, $\partial_{+} M$, where $\partial_{+} M\neq\emptyset$. 
\end{notation}

\begin{definition}
A subpolyhedron $P$ of a 3-manifold $M\in\mathcal{T}$ is called a \emph{skeleton of} $(M, \partial_{-} M)$ if $P\cup \partial M$ is simple, $M \searrow (P\cup \partial_{-} M)$ and $P$ intersects each component $T$ of $\partial M$ either in a theta-curve, or in a nontrivial simple closed curve, or not at all.
\end{definition}

\begin{remark}
Note that if $\partial_{-} M=\emptyset$ and $P\cap \partial_{+} M = \emptyset$ then a skeleton $P$ of $(M, \partial_{-} M)$ is a simple spine of $M$.
\end{remark}

Now we describe a fundamental operation on the set of manifolds from $\mathcal{T}$ and their skeleta. Let $M_1$, $M_2$ be two manifolds from $\mathcal{T}$. Let $P_i$, $i=1,2$, be a skeleton of $(M_i, \partial_{-} M_i)$. Choose two tori $T_1 \subset \partial_{+} M_1$, 
$T_2 \subset \partial_{-} M_2$ such that $P_1\cap T_1$ is nonempty and homeomorphic to $P_2\cap T_2$.
Fix a homeomorphism $\varphi\co T_1\to T_2$ taking $P_1\cap T_1$ to $P_2\cap T_2$.
We can then construct a new manifold $W = M_1\cup_\varphi M_2$ with 
$\partial_{+} W = (\partial_{+} M_1 - T_1) \cup \partial_{+} M_2$ and 
$\partial_{-} W = \partial_{-} M_1 \cup (\partial_{-} M_2 - T_2)$.
Then $P=P_1\cup P_2$ is a skeleton of $(W, \partial_{-} W)$.
We say that the manifold $W\in \mathcal{T}$ described above is
obtained by \emph{assembling} $M_1$ and~$M_2$. The same terminology
is used for skeleta: $P$ is obtained by assembling $P_1$ and $P_2$.


\subsection{Examples}\label{SS:DefBlocks}

In this subsection we give five examples of 3-manifolds and their skeleta.
These five pairs will be the building blocks which we 
shall use later to construct spines of arbitrary Seifert fibered 3-manifolds with boundary.

\begin{example}\label{Block1} (The main block)
Consider an orientable $S^1$-bundle $M_0\in \mathcal{T}$ over a compact 
connected surface $F_0$. Suppose that $\partial_-M_0\neq\emptyset$ and $F_0$ is not an annulus. The subdivision of $\partial M_0$ into $\partial_{+} M_0$
and $\partial_{-} M_0$ induces a subdivision of~$\partial F_0$ into $\partial_{+} F_0$ and $\partial_{-} F_0$. By gluing a disk to each component of $\partial_{-} F_0$, we obtain a surface, still with nonempty boundary, 
which we call~$F$. There is an obvious embedding $F_0\hookrightarrow F$. 

As described in the proof of Proposition~\ref{P:NoSingularFibre}, there exists a graph $\Gamma\subset F$ such that $F-\Gamma\cong \partial F\times [0,1)$. After an isotopy we can assume that the graph~$\Gamma$ is disjoint from the glued-in disks, i.e.~that $\Gamma\subset F_0$. 

Now we construct a properly embedded graph $\Gamma_0$ in $F_0$ as follows. If $\Gamma$ has at least one edge, then we perform, for each circle $c\subset\partial_{-} F_0$ in turn, the following operation: 
we add a new vertex on $c$, subdivide some edge of the graph by adding a new vertex in its interior, and add a new edge connecting the two newly created vertices. Each vertex of the resulting graph~$\Gamma_0$ has valence at most $3$. If, on the other hand, $\Gamma$ consists of only one vertex, then $F$ is a disk and $\partial_{+}F_0=S^1$. 
Moreover, $\partial_{-}F_0$ contains at least two circles, because $F_0$ is not an annulus. In this case, the above algorithm for constructing $\Gamma_0$ has to be adapted: for the first two circles $c_1, c_2\subset\partial_{-} F_0$, the newly created vertices on~$c_1, c_2$ have to be connected by an edge to the single vertex of~$\Gamma$. For the remaining boundary circles in $\partial_-F_0$ we then proceed as before.

Let $p\co M_0\to F_0$ be the projection map. Then $P_0=p^{-1}(\Gamma_0)$ is a skeleton of~$(M_0, \partial_{-} M_0)$. We shall call the pair $(M_0, P_0)$ the \emph{main block}. We note that there are no true vertices in the simple polyhedron $P_0\cup \partial M_0$.
\end{example}

\begin{example}\label{Block2} (The first solid torus block)
Let $V$ be a solid torus. We shall present $\partial V$ as the union of $\partial_{+} V = \partial V$ and $\partial_{-} V = \emptyset$. Let~$D$ be a meridional disk of~$V$. Choose a properly embedded M\"obius strip which intersects~$D$ in one arc. This arc cuts~$D$ into two halves. Denote by~$P_V$ the union of the M\"obius strip and one half of~$D$. Then $P_V$ is a skeleton of~$(V, \partial_{-} V)$. The pair $(V,P_V)$ will be our next building block, which we call the \emph{first solid torus block} -- see Figure~\ref{F:Blocks}(a). We note that there are no true vertices in the simple polyhedron $P_V\cup \partial V$.
\end{example}

\begin{example}\label{Block2'} (The second solid torus block)
Let $V$, $\partial_{+} V$ and $\partial_{-} V$ be as in Example~\ref{Block2}. Then a properly embedded M\"obius strip $P_{\textrm{M\"ob}}$ (which intersects a meridional disk in one arc) is a skeleton of~$(V, \partial_{-} V)$. 
The pair $(V,P_\textrm{M\"ob})$ will be our next building block, which we call the \emph{second solid torus block}. We note that there are no true vertices in the simple polyhedron $P_\textrm{M\"ob}\cup \partial V$.
\end{example}

\begin{figure}[htb]
\centering
\includegraphics[scale=0.6]{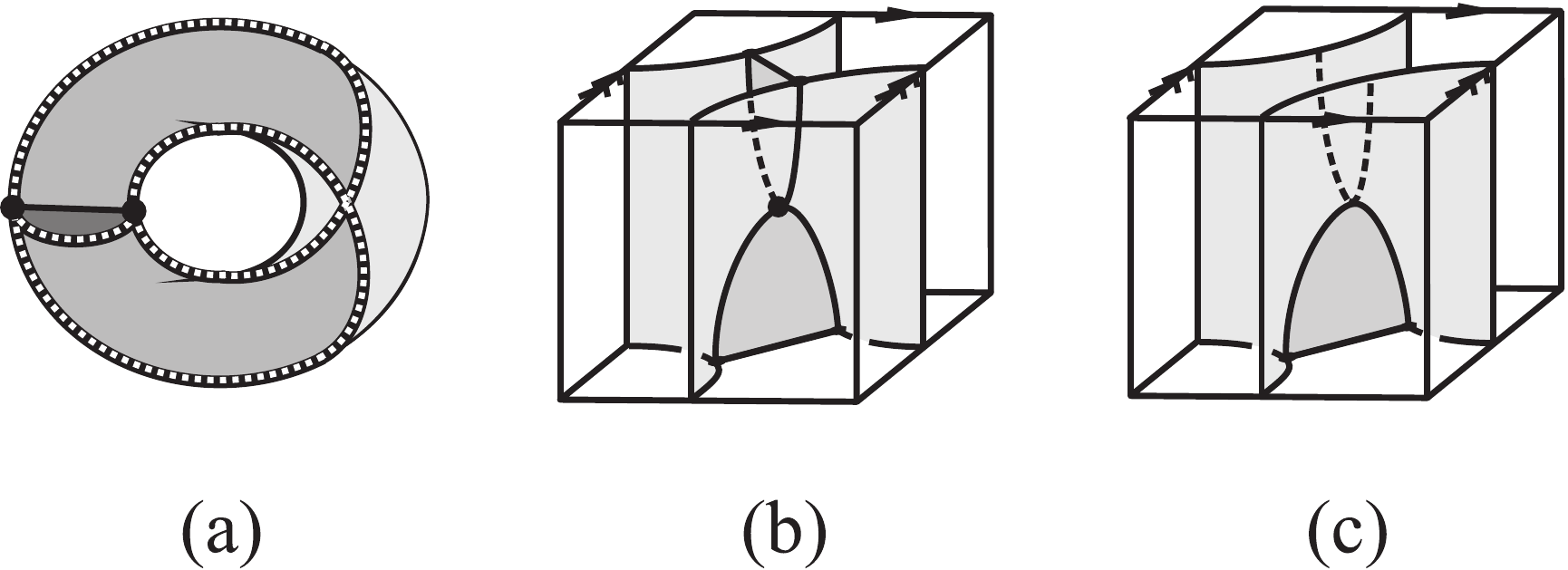}
\caption{Examples of the building blocks: (a) the first solid torus block, \ (b) the flip block, \ (c) the transitional block} \label{F:Blocks}
\end{figure}

\begin{definition}
Let $\theta$ be a theta-curve on a torus~$T$, and let~$\gamma$ be one of the three edges of~$\theta$. A \emph{flip} of~$\theta$ along~$\gamma$ is the operation of replacing $\theta$ by a new theta-curve~$\theta'$ as follows: shrink the edge~$\gamma$ so as to obtain a graph with only two edges and a single four-valent vertex, and re-expand in the other direction.  
\end{definition}

Thus a flip is just a Whitehead-move of a theta-curve (see Fig.~\ref{flip-transformation}). It is well known that any two  theta-curves on~$T$ can be transformed into each other by 
a sequence of flips. For more details on this, see Section~\ref{SS:theta-curves} below.

\begin{figure}[htb]
\centering
\includegraphics[scale=0.6]{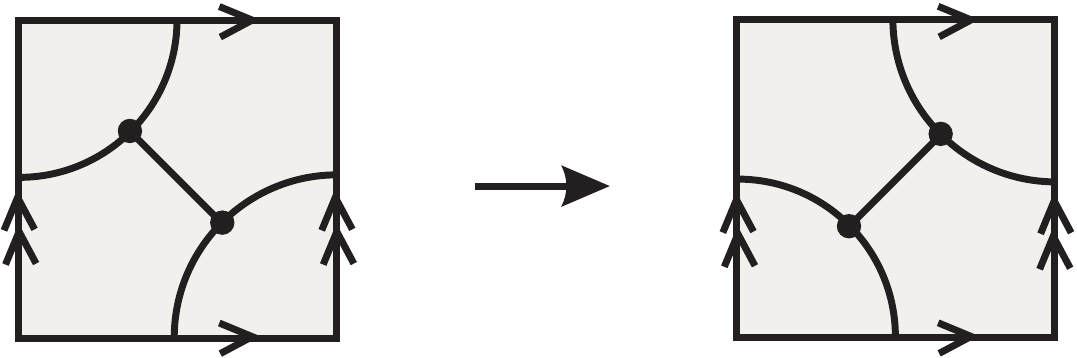}
\caption{A flip-transformation} \label{flip-transformation}
\end{figure}

\begin{example}\label{Block3} (The flip block)
Let $\theta$, $\theta'$ be theta-curves on
a torus $T$ such that $\theta'$ is obtained from $\theta$ by exactly
one flip. Suppose that $M = T\times [0,1]$, $\partial_{+} M = T\times\{1\}$ and $\partial_{-} M = T\times\{0\}$. Then there is a skeleton $P$ of $(M, \partial_{-} M)$, 
such that 
\begin{itemize}
\item[(1)] for each $t\in [0, 1/2)$ the theta-curve $\theta_t$ is isotopic to $\theta$, where~$\theta_t$ is defined by $P\cap \left(T\times \{t\}\right) = \theta_t\times \{t\}$, ;
\item[(2)] for each $t\in (1/2, 1]$ the theta-curve $\theta_t$ is isotopic 
to~$\theta'$;
\item[(3)] $P\cap \left(T\times \{1/2\}\right)$ is a wedge of two circles.
\end{itemize}
See Figure~\ref{F:Blocks}(b), where the torus $T$ is represented as a square with identified sides.
Note that $P\cap (T\times \{t\}) = \theta_t\times \{t\}$, where $t$ varies from $0$ to 1, yields a movie of a flip transforming $\theta$ into $\theta'$. 
Also note that the simple polyhedron~$P\cup \partial M$ has exactly one true vertex. The pair $(T\times [0, 1], P)$ is our fourth building block, called the \emph{flip block}.
\end{example}

\begin{example}\label{Block4} (The transitional block) We recall that every theta-curve contains exactly three simple closed curves, each obtained by removing one of the three edges.
Now let $\theta\subset T$ be a theta-curve and $\ell\subset T$ a simple closed curve on a torus such that 
\begin{enumerate}
\item $\ell$ is not isotopic to any of the simple closed curves contained in $\theta$, but 
\item there exists a theta-curve $\theta'\subset T$ containing $\ell$ which is obtained from $\theta$ by a single flip.
\end{enumerate}
Let $M$, $\partial_{+} M$ and $\partial_{-} M$ be as in Example~\ref{Block3}. Then we can construct (see Figure~\ref{F:Blocks}(c)) a skeleton $P$ of $(M, \partial_{-} M)$ such that 
\begin{itemize}
\item[(1)] for each $t\in [0, 1/2)$ the theta-curve $\theta_t$ defined by $P\cap \left(T\times \{t\}\right) = \theta_t\times \{t\}$ is isotopic to
$\theta$;
\item[(2)] for each $t\in (1/2, 1]$ the simple closed curve $\ell_t$ defined by
$P\cap \left(T\times \{t\}\right) = \ell_t\times \{t\}$ is isotopic to~$\ell$;
\item[(3)] $P\cap (T\times \{1/2\})$ is a wedge of two circles.
\end{itemize}
We note that the simple polyhedron $P\cup \partial M$ has no true vertex.
The pair $(T\times [0, 1], P)$ is our last building block. We call it the \emph{transitional block}, because it has a simple closed curve on one boundary component and a theta-curve on the other.
\end{example}


\subsection{Theta-curves on a torus}\label{SS:theta-curves}

Let us endow the set $\Theta(T)$ of isotopy classes of theta-curves on~$T$ 
with the distance function~$d$ which is defined as follows: 
for $\theta, \theta'\in \Theta(T)$, the distance 
$d(\theta, \theta')$ is the minimal number
of flips required to transform~$\theta$ into~$\theta'$.

 Fix some coordinate system $(\mu, \lambda)$ on a torus $T$. Note that each 
theta-curve $\theta$ on $T$ contains three nontrivial simple closed curves which are formed by the pairs of edges of $\theta$.

\begin{notation} 
 Let us denote by $\theta^b$ the theta-curve on $T$ containing simple closed curves 
$\mu, \lambda, \mu\lambda$, and for every coprime positive integers $\alpha, \beta$ denote by $\theta(\alpha/\beta)$ the theta-curve on $T$ which is closest to $\theta^b$ among all the theta-curves on $T$ containing the simple closed curve $\mu^{\alpha}\lambda^{\beta}$. 
\end{notation}

\begin{lemma}\label{L:S(p,q)forTheta}
For every coprime positive integers $\alpha, \beta$ the distance between $\theta^b$
and $\theta(\alpha/\beta)$ is equal to $S(\alpha,\beta)-1$.
\end{lemma}

\begin{proof}
 For calculating the distance between theta-curves on a torus
we use the classical ideal Farey triangulation $\mathbb{F}$ of the
hyperbolic plane $\mathbb{H}^2$. If we view the hyperbolic plane
$\mathbb{H}^2$ as the upper half plane of $\mathbb{C}$ bounded by
the circle $\partial \mathbb{H}^2 = \mathbb{R}\cup \{\infty\}$, then
the triangulation $\mathbb{F}$ has vertices at the points of
$\mathbb{Q}\cup \{1/0\}\subset \partial \mathbb{H}^2$, where
$1/0=\infty$, and its edges are the geodesics in $\mathbb{H}^2$ with
endpoints the pairs $a/b$, $c/d$ such that $ad-bc=\pm 1$. For
convenience, the image of the hyperbolic plane $\mathbb{H}^2$ and
the triangulation $\mathbb{F}$ under the mapping $z\to (z-i)/(z+i)$
are shown in Fig.~\ref{triangulation}.

\begin{figure}[htb]
\centering
\includegraphics[scale=0.5]{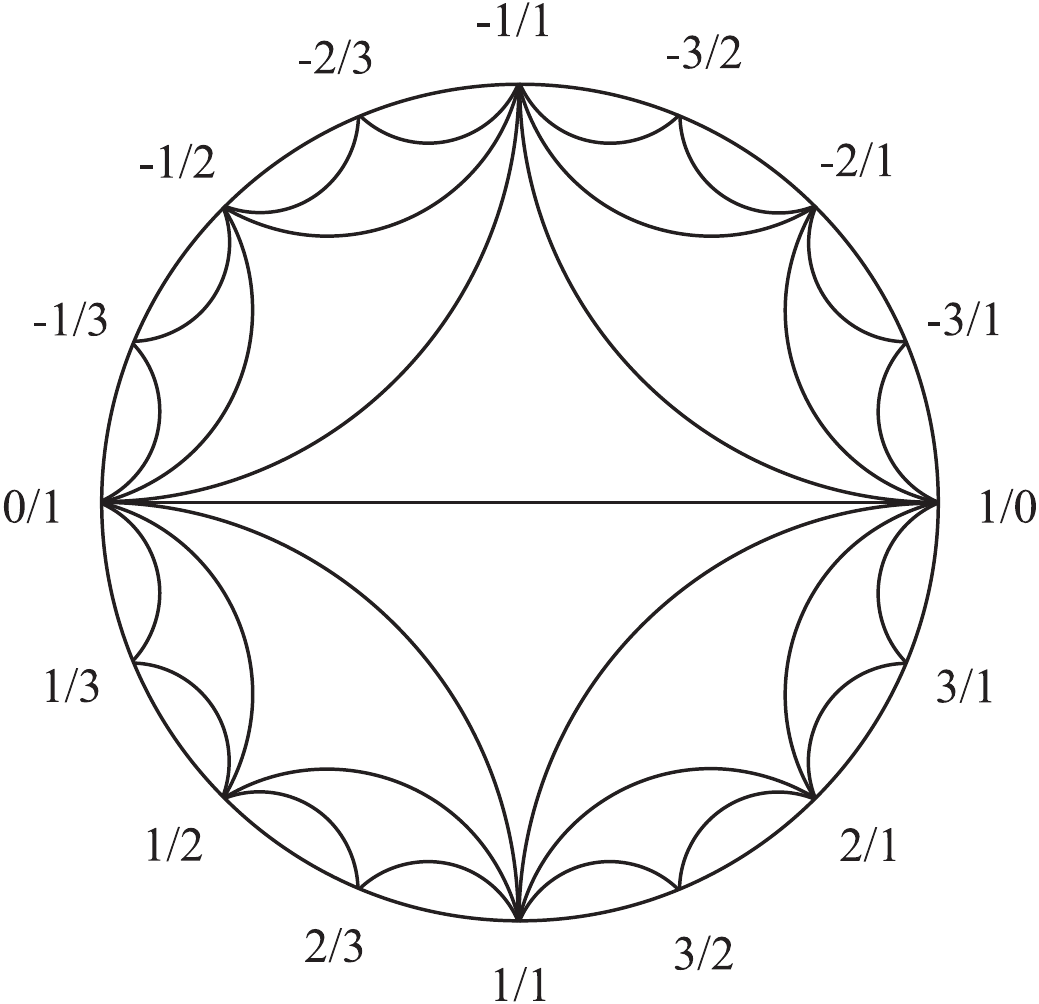}
\caption{The ideal triangulation of the hyperbolic plane}
\label{triangulation}
\end{figure}

 Let us construct a map $\Psi_{\mu, \lambda}$ from $\Theta(T)$ to the set of
triangles of $\mathbb{F}$. To do that, we consider the map
$\psi_{\mu, \lambda}$ assigning to each nontrivial simple closed
curve $\mu^{\alpha}\lambda^{\beta}$ on $T$ the point $\alpha/
\beta\in \partial\mathbb{H}^2$. Consider a  theta-curve $\theta$ on $T$, and its three nontrivial simple closed curves denote by $\ell_1$, $\ell_2$, $\ell_3$. Since the intersection index of any two curves $\ell_i$, $\ell_j$, $i\neq j$, is equal to $\pm 1$, the points $\psi_{\mu,\lambda}(\ell_1)$, $\psi_{\mu, \lambda}(\ell_2)$, $\psi_{\mu,
\lambda}(\ell_3)$ are the vertices of a triangle $\triangle$ of the
Farey triangulation. So we define $\Psi_{\mu, \lambda}(\theta)$ to
be $\triangle$.

 Denote by $\Sigma$ the graph dual to the triangulation
$\mathbb{F}$. This graph is a tree because the triangulation is
ideal. We now define the distance between any two triangles of
$\mathbb{F}$ to be the number of edges of the only simple path in
$\Sigma$ that joins the corresponding vertices of the dual graph.
The key observation used for the practical calculations is that for
any coordinate system $(\mu, \lambda)$ on $T$ the distance between
theta-curves $\theta$, $\theta'$ is equal to the distance between
the triangles $\Psi_{\mu, \lambda}(\theta)$, $\Psi_{\mu,
\lambda}(\theta')$ of the Farey triangulation. The reason is that 
$\theta'$ is obtained from $\theta$ via a flip if and only if the 
corresponding triangles have a common edge.

Let us denote by
$\sigma(\alpha_1/\beta_1, \alpha_2/\beta_2, \alpha_3/\beta_3)$ the
triangle of the Farey triangulation with vertices $\alpha_1/\beta_1$, 
$\alpha_2/\beta_2$, $\alpha_3/\beta_3$, and by $\sigma(\alpha/\beta)$ 
denote the closest triangle to the base triangle $\sigma(0/1, 1/0, 1/1)$ among all 
the triangles with the vertex $\alpha/\beta$. 

 It is easy to see that $\Psi_{\mu, \lambda}(\theta^b) = \sigma(0/1, 1/0, 1/1)$ and $\Psi_{\mu, \lambda}(\theta(\alpha/\beta)) = \sigma(\alpha/\beta)$. Hence the distance between $\theta^b$ and $\theta(\alpha/\beta)$ is equal to the distance between the triangles $\sigma(0/1, 1/0, 1/1)$ and $\sigma(\alpha/\beta)$. Finally the fact that the distance between the triangles $\sigma(0/1, 1/0, 1/1)$
and $\sigma(\alpha/\beta)$ is equal to $S(\alpha,\beta)-1$ is well-known, see for instance \cite[Chapter II.4]{Dalbo} or \cite[Chapter 9]{Indra}. For completeness, a proof will be given in Lemma~\ref{L:S(p,q)forTriangles}. 
\end{proof}

We now turn our attention back to Example~\ref{Block2}. Recall that a skeleton $P_V$ of~$(V, \partial_{-} V)$ intersects~$\partial V$ in a theta-curve~$\xi_V$
(see Figure~\ref{F:Blocks}(a)) -- this theta-curve consists of the boundary of the M\"obius strip and of one arc of~$\partial D$ which we shall now call~$\gamma$.
Note that among the three nontrivial simple closed curves lying in
$\xi_V$ there is no one which is isotopic to the meridian
$m=\partial D$ of~$V$. However, applying the flip along~$\gamma$ 
to~$\xi_V$ yields a theta-curve $\xi_m\subset \partial V$
containing~$m$. This implies the following lemma:

\begin{lemma}\label{L:Exist}
Among all homeomorphisms $\partial V\to T$ that take the meridian $m$ of $V$ to the curve $\mu^{\alpha}\lambda^{\beta}$ there exists a homeomorphism $h$ such that $h(\xi_m) = \theta(\alpha/\beta)$ and $h(\xi_V)$ is $d$-distance one closer to $\theta^b$ than $\theta(\alpha/\beta)$.
\end{lemma}


\subsection{Proof of Theorem~\ref{T:ComplSFS}}

Given the 3-manifold $M=\big(F, (p_1,q_1),\ldots ,(p_k,q_k)\big)$, our 
strategy for proving Theorem~\ref{T:ComplSFS} is to realize $M$ as an assembling of several copies of the five building blocks introduced in
Section~\ref{SS:DefBlocks}, being as economical as possible with 
blocks that contain skeleta with true vertices. More specifically, we shall use exactly one main block, and for each of the~$k$ singular fibers we use one solid torus block (first or second), one transitional block and $\max\{S(p_i, q_i)-3, 0\}$ flip blocks.
Now we explain the details of this plan.

{\bf Step 1 } Let $p,q$ be coprime integers such that $p>q>0$ and $p/q\neq 2/1$. Consider a torus $T$ with a fixed coordinate system $(\mu,\lambda)$.
Let $\theta\subset T$ be the theta-curve that is $d$-distance one closer to $\theta^b$ than $\theta(p/q)$. Suppose that $N = T\times [0,1]$, $\partial_{+} N = T\times\{1\}$ and $\partial_{-} N = T\times\{0\}$. Our aim now is to construct a skeleton $\Pi$ of 
$(N, \partial_{-} N)$ satisfying the following conditions.
\begin{enumerate}
\item $\Pi\cap \partial_{-} N=\theta\times\{0\}$
\item $\Pi\cap \partial_{+} N=\lambda\times\{1\}$
\item the simple polyhedron $\Pi\cup \partial N$ has $S(p,q)-3$ true vertices.
\end{enumerate}

 It follows from Lemma~\ref{L:S(p,q)forTheta} that $d(\theta,\theta^b)=S(p,q)-2$. Thus there is a sequence of theta-curves in $T$
$$
\theta=\theta_0\to\theta_1\to\ldots\to\theta_{r}=\theta^b, \hbox{ \ where } r=S(p,q)-2
$$
such that each $\theta_{j+1}$ is obtained from $\theta_j$ by one flip (see left hand side of Figure~\ref{F:GluePieces}). Since $p/q\neq 2/1$ we have $r\geqslant 1$.

 For the 3-manifold $N_r=T\times[\frac{r-1}{r},1]$ with $\partial_{+} N_r = T\times\{1\}$ and $\partial_{-} N_r = T\times\{\frac{r-1}{r}\}$ we construct a skeleton $P_{r}$, as described in Example~\ref{Block4} (the transitional block), such that $P_{r}\cap \partial_{+} N_r=\lambda\times\{1\}$ and $P_{r}\cap \partial_{-} N_r=\theta_{r-1}\times\{\frac{r-1}{r}\}$.  If $r=1$ we define $\Pi=P_{r}$.

 If $r > 1$, then for each manifold $N_j=T\times[\frac{j-1}{r},\frac{j}{r}]$ (for $j=1,\ldots,r-1$) with $\partial_{+} N_j = T\times\{\frac{j}{r}\}$ and $\partial_{-} N_j = T\times\{\frac{j-1}{r}\}$ we construct a skeleton $P_{j}$, as described in Example~\ref{Block3} (the flip block), such that $P_{j}\cap \partial_{+} N_j=\theta_{j}\times\{\frac{j}{r}\}$ and $P_{j}\cap \partial_{-} N_j=\theta_{j-1}\times\{\frac{j-1}{r}\}$. 
 
 Finally we can realize $N$ as an assembling of the manifolds $N_1,\ldots, N_{r}$ by means of the identity homeomorphisms on the tori $T\times\{\frac{j}{r}\}$, for $j=1,\ldots, r-1$. Then we define the skeleton $\Pi$ to be the result of an assembling of all the polyhedra $P_j$. 

\begin{figure}[htb]
\centerline{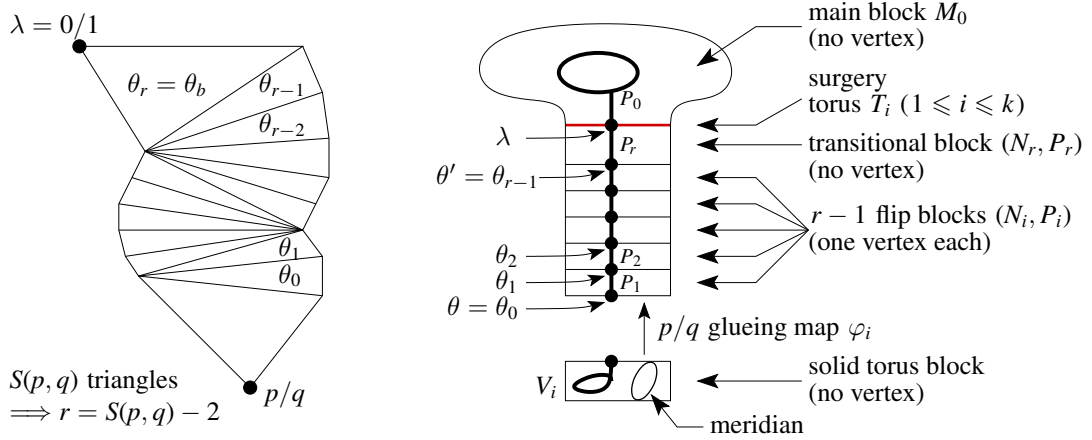}
\caption{Left: a picture in the Farey graph. Right: a symbolic picture of the 3-manifold. The polyhedron is drawn in bold. For simplicity, only one surgery is shown, instead of~$k$.} 
\label{F:GluePieces}
\end{figure}

{\bf Step 2 } Now we construct a simple spine for the Seifert manifold $\big(F, (p_1,q_1),\ldots ,(p_k,q_k)\big)$.

If $k=0$, or $F$ is a disk and $k=1$, this was achieved in Proposition~\ref{P:NoSingularFibre}. Thus from now on we assume that $k\geqslant 1$, and if $F$ is a disk then $k\geqslant 2$. 

We use the notation for the Seifert manifold from the Introduction. Consider a surface $F_0$ obtained from $F$ by removing the interiors of $k$ disjoint disks. The boundary circles of these disks are denoted by $c_1, \ldots, c_k$. Let $c_{k+1}, \ldots, c_n$ be all the remaining circles of~$\partial F$. 
Consider an orientable $S^1$-bundle $M_0$ over $F_0$. Let the boundary tori $T_i = p^{-1}(c_i)$, for $1\leqslant i\leqslant k$, form $\partial_{-} M_0$ and $T_i = p^{-1}(c_i)$, for $k+1\leqslant i\leqslant n$, form $\partial_{+} M_0$.
Now we can construct a skeleton $P_0$ of~$(M_0, \partial_{-} M_0)$ as described in Example~\ref{Block1} (the main block). Note that on each torus $T_i$ we have a coordinate system $(\mu_i, \lambda_i)$ and $P_0\cap T_i = \lambda_i$, for $1\leqslant i\leqslant k$. 

 For each $i$, $1\leqslant i\leqslant k$, we do the following depending on the value of $p_i/q_i$. 

 Case 1. $p_i/q_i\neq 2/1$. Suppose that $N_i = T_i\times [0,1]$, $\partial_{+} N_i = T_i\times\{1\}$ and $\partial_{-} N_i = T_i\times\{0\}$. Construct a skeleton $\Pi_i$ of $(N_i, \partial_{-} N_i)$ as explained in Step 1. Take a solid torus $V_i$ with its skeleton $P_{V_i}$ as described in Example~\ref{Block2} (the first solid torus block). We recall that~$P_{V_i}$ intersects~$\partial V_i$ in a theta-curve~$\xi_{V_i}$. Among all homeomorphisms $\partial V_i\to \partial_{-} N_i$ that take the meridian $m_i$ of $V_i$ to the curve $\mu^{p_i}\lambda^{q_i}\times\{0\}$ we choose a homeomorphism $h_i$ such that $h_i(\xi_{V_i}) = \Pi_i\cap \partial_{-} N_i$. The possibility of such a choice is justified by Lemma~\ref{L:Exist}. Then we assemble the manifolds $V_i, N_i, M_0$ and their skeleta $P_{V_i}, \Pi_i, P_0$ by homeomorphisms $h_i$ and $\phi_i\co\partial_{+} N_i\to T_i$, where $\phi_i$ takes each pair $(x,1)\in T_i\times\{1\}$ to the point $x\in T_i$.

 Case 2. $p_i/q_i=2/1$. Take a solid torus $V_i$ with its skeleton $P_{V_i}$ as described in Example~\ref{Block2'} (the second solid torus block). We recall that~$P_{V_i}$ intersects~$\partial V_i$ in a simple closed curve $\ell_i$. 
Among all homeomorphisms $\partial V_i\to T_i$ that take the meridian $m_i$ of $V_i$ to the curve $\mu_i^2\lambda_i$ we choose a homeomorphism $h_i$ such that $h_i(\ell_i) = \lambda_i$. To finish off the construction we assemble the manifolds $V_i, M_0$ and their skeleta $P_{V_i}, P_0$ by the homeomorphism $h_i$.

Finally we observe that $\partial_{-} M=\emptyset$ and then the constructed skeleton $P$ of $(M, \partial_{-} M)$ is a simple spine of $M$. Moreover, by construction $P$ has $\sum_{i=1}^k \max\{S(p_i, q_i)-3, 0\}$ true vertices and hence
 $$c(M)\leqslant \sum_{i=1}^k \max\{S(p_i, q_i)-3, 0\}.$$
 

\section{The case of torus knots}
\label{S:TorusKnots}

\subsection{Torus knots and their complexity}

We recall the definition of the $(\alpha,\beta)$-torus knot
$T(\alpha,\beta)$ -- we will always assume $\alpha>\beta$, and that
$\alpha$ and $\beta$ are coprime. First let $T^2$ be an unknotted
torus in $S^3$; this cuts $S^3$ into two components. Let $[m],[l]\in
H_1(T^2)$ be the elements represented by simple closed curves on
$T^2$ which are contractible in the inner and the outer component,
respectively. Now the torus knot $T(\alpha,\beta)$ is the knot which
lies in $T^2$ and which represents  the element $[m]^\alpha [l]^\beta$ in
$H_1(T^2)$.

In other words, the $(\alpha,\beta)$-torus knot is the closure of
the $\beta$-strand braid with $\alpha(\beta-1)$ crossings shown
in Figure~\ref{F:TorusKnot}(a).

\begin{figure}[htb]
\centerline{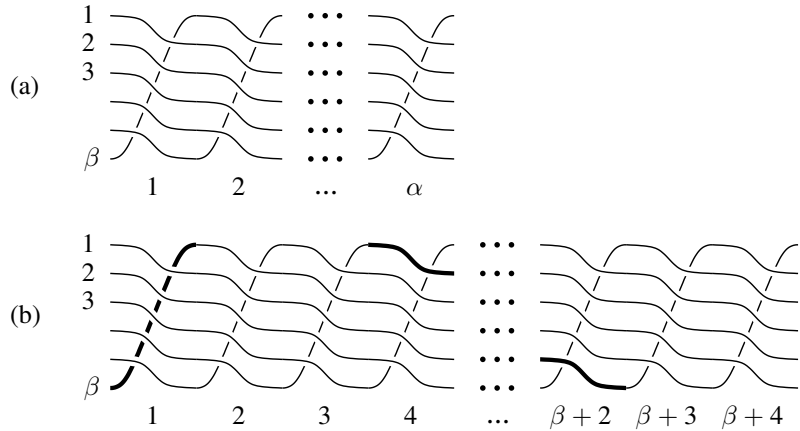}
\caption{(a) The $(\alpha,\beta)$-torus knot is the closure of this
braid. \ (b) Independent over- and underpasses of torus knots when $\alpha\geqslant \beta+4$.} \label{F:TorusKnot}
\end{figure}

We recall our second main result 

{\bf Theorem~\ref{T:ComplTorusKnots} } {\sl The complexity of the complement of the $(\alpha,\beta)$-torus knot, for $\alpha>\beta\geqslant 2$, is at most} 
$$\textstyle{\max\left\{S(\alpha,\beta)-3\ , \
0\right\}+\max\left\{S(\alpha,\beta)-\lfloor\frac{\alpha}{\beta}\rfloor-3
\ , \ 0\right\}}$$

Before proving this result, we deduce some corollaries.

\begin{corollary}\label{C:SharpForSimplest}
(a) The complexity of the complement of the $(3,2)$-torus knot (trefoil) is~0.

(b) The complexity of the complements of the torus knots with parameter $(5,2)$, $(4,3)$ and $(5,3)$ is equal to 1.
\end{corollary}

\begin{proof}
(a) The upper bound from Theorem~\ref{T:ComplTorusKnots} is zero in the case of the trefoil.

(b) In these cases, the upper bound from Theorem~\ref{T:ComplTorusKnots} is equal to one. According to the classification of irreducible, boundary-irreducible orientable 3-manifolds with nonempty boundary of complexity zero in \cite{Nik08}, their complexity is not zero.
\end{proof}

\begin{remark} The bound in Theorem~\ref{T:ComplTorusKnots} is a little bit tedious to calculate by hand for larger values of $\alpha$ and $\beta$.
A Mathematica-command for
calculating a table of values of this function (with $\alpha$
vertically and $\beta$ horizontally, both starting from $1$) is

{\tt Grid[Table[Boole[(a>b)]*(Max[Total[ContinuedFraction[a/b]]-3,0] +}\\
{\tt
Max[Total[ContinuedFraction[a/b]]-Floor[a/b]-3,0]),{a,50},{b,50}]] }

Some rough upper bounds which are convenient to calculate are given
in the next corollary.
\end{remark}

\begin{corollary}\label{C:SimpleBounds} (a) For 
$\alpha\geqslant 3$, the complexity of the complement of the $(\alpha,\alpha-1)$-torus knot is at most  $\max(2\alpha-7,0)$.

(b) For $\alpha$, $\beta$ two coprime integers with $2\leqslant\beta\leqslant
\alpha-2$, the complexity of the complement of the
$(\alpha,\beta)$-torus knot is at most $\alpha-5$ (if $\alpha$ is even) or $\alpha-4$ (if $\alpha$ is odd). 
\end{corollary}

\begin{proof}
(a) By Theorem~\ref{T:ComplTorusKnots}, the complexity of the $(\alpha,\alpha-1)$-torus knot is at most 
$\max\left\{S(\alpha,\alpha-1)-3, 
0\right\}+\max\left\{S(\alpha,\alpha-1)-
\lfloor\frac{\alpha}{\alpha-1}\rfloor-3, 0\right\}$. 
The result now follows from the fact that $S(\alpha,\alpha-1)=\alpha$, and that $\lfloor\frac{\alpha}{\alpha-1}\rfloor= 1$. 

(b) We need some estimates on $S(\alpha,\beta)$ (for $\alpha$ and $\beta$ 
coprime, $2\leqslant \beta\leqslant\alpha-2$), depending on the parity of~$\alpha$. 

We claim that if $\alpha$ is even, then $S(\alpha,\beta)\leqslant \frac{\alpha}{2}+1$ (this will be proven below). Assuming this claim, it follows that $S(\alpha,\beta)-3\leqslant \frac{\alpha}{2}-2$ and $S(\alpha,\beta)-3-\lfloor\frac{\alpha}{\beta}\rfloor\leqslant \frac{\alpha}{2}-3$ (using the fact that $\lfloor\frac{\alpha}{\beta}\rfloor\geqslant 1$).
This yields an upper bound on the complexity of the knot complement of $\frac{\alpha}{2}-2+\frac{\alpha}{2}-3=\alpha-5$, as desired.

If $\alpha$ is odd, then we claim $S(\alpha,\beta)\leqslant \frac{\alpha+3}{2}$. Again assuming this claim, the complexity is bounded by $\frac{\alpha+3}{2}-3+\frac{\alpha+3}{2}-4=\alpha-4$, which is what we wanted. 

We shall now prove the preceding two claims. Actually, we shall only prove them under the additional hypothesis that $\alpha$ is even and $\alpha\geqslant 12$, or that $\alpha$ is odd and $\alpha\geqslant 9$. (The claims are easy to check by hand if $\alpha$ is even and $\alpha\leqslant 10$, or if $\alpha$ is odd and $\alpha\leqslant 7$.) We first notice that, since $S(\alpha,\beta)=S(\alpha,\alpha-\beta)$, it is sufficient to prove the claims for $\beta\leqslant\frac{\alpha}{2}$; in fact, we can even assume that $\beta<\frac{\alpha}{2}$, since $\alpha$ and $\beta$ are coprime. Now the key observation is that 
$$
\textstyle S(\alpha,\beta)\leqslant \beta+\left\lfloor \frac{\alpha}{\beta}\right\rfloor
$$
So we have to prove that for all relevant values of $\beta$,
$$
\textstyle \beta+\left\lfloor \frac{\alpha}{\beta}\right\rfloor\leqslant \frac{\alpha}{2}+1 \text{\ \ (}\alpha\text{\ even) \ \ \ \ \ and \ \ \ \ \ } \textstyle \beta+\left\lfloor \frac{\alpha}{\beta}\right\rfloor\leqslant \frac{\alpha+3}{2} \text{\ \ (}\alpha\text{\ odd)}
$$
Now if $\alpha$ is even, then $\beta\neq 2$ (since $\beta$ cannot divide $\alpha$), and we only have to deal with $\beta\in \{3,\ldots,\frac{\alpha}{2}-1\}$. If $\beta=\frac{\alpha}{2}-1$, then 
$$\textstyle \beta+\left\lfloor \frac{\alpha}{\beta}\right\rfloor =\frac{\alpha}{2}-1+\left\lfloor \frac{2\alpha}{\alpha-2}\right\rfloor = \frac{\alpha}{2}+1$$ 
where the last equality is true if $\alpha>6$. Thus it suffices to prove that for~$\beta$ in the interval $[3,\frac{\alpha}{2}-2]$ the convex function $\beta\mapsto\beta+\frac{\alpha}{\beta}$ is bounded above by $\frac{\alpha}{2}+1$ (at least assuming that $\alpha\geqslant 12$). We leave this as an easy exercice to the reader. This completes the case where $\alpha$ is even.

If $\alpha$ is odd, then the proof is similar. We have to deal with the values $\beta\in\{2,\ldots,\frac{\alpha-1}{2}\}$. If $\beta=2$ or $\beta=\frac{\alpha-1}{2}$, then a calculation shows that $\beta+\left\lfloor \frac{\alpha}{\beta}\right\rfloor = \frac{\alpha+3}{2}$ 
Now it suffices to prove that on the interval $[3,\frac{\alpha-3}{2}]$ the convex function $\beta\mapsto\beta+\frac{\alpha}{\beta}$ is bounded above by $\frac{\alpha+3}{2}$, provided $\alpha\geqslant 9$. Again, this is left as an easy exercice to the reader. 
\marg{\textcolor{white}{B: Numerical calculations suggest that for or $\alpha$, $\beta$ two coprime integers with $\alpha$ even and $3\leqslant\beta\leqslant\alpha-3$, our bound on the complexity of the complement of the $(\alpha,\beta)$-torus knot is at most $\lfloor \frac{2\alpha-5}{3}\rfloor$.}}
\end{proof}


\subsection{Comparison with Matveev's bound}

In \cite[Proposition 2.1.11]{Mat03}, Matveev gives an upper bound for the complexity of the complement space of any link in $S^3$. We shall see that at least in the case of torus knots, his bound is far from sharp.

Starting from a knot diagram with $n$ crossings, with an overpass of length~$k$ and an independent underpass of length~$m$, Matveev manages to construct explicitly an almost simple spine with $4(n-m-k-2)$ true vertices. Here the word \emph{independent} means that the underpass and the overpass, including their endpoints, are disjoint. If over- and underpass are not independent, then an extra correction term must be added to Matveev's bound. In the case of torus knots one obtains the following bounds:

\begin{proposition}[Complexity bounds for torus knots from Matveev's construction] (a) If $\alpha\geqslant \beta+4$, then the complexity of the 
complement of the $(\alpha,\beta)$-torus knot satisfies
$$
c(S^3-T(\alpha,\beta))\leqslant 4(\alpha\beta-\alpha-2\beta)
$$
(b) For the trefoil we have $c(S^3-T(3,2))\leqslant 1$

(c) For the $(5,2)$-torus knot we have $c(S^3-T(5,2))\leqslant 5$
\end{proposition}

\begin{proof} As can be seen in Figure~\ref{F:TorusKnot}(b), if $\alpha\geqslant \beta+4$, then there are independent over- and underpasses of maximal length $\beta-1$. Thus, according to Matveev, the complexity is at most $4(\alpha(\beta-1)-2(\beta-1)-2)=4(\alpha\beta-\alpha-2\beta)$. The proofs of (b) and (c) are left as an exercise to the reader.
\end{proof}

We see that for torus knots the bound in Theorem~\ref{T:ComplTorusKnots} is more powerful than Matveev's general bound. Indeed, for large values of $\alpha$ and $\beta$, the Matveev bound is $O(\alpha\beta)$, whereas ours is $O(\alpha)$. Also, from Theorem~\ref{T:ComplTorusKnots} we obtain $c(E(3,2))=0$ and $c(E(5,2))\leqslant 1$.

Our results should also be compared to those of by Casali and Cristofori \cite{Casali}, which also yield an upper bound for Matveev's complexity of torus knot complements via crystallization theory~\cite[Prop.\ 13]{Casali}. It turns out that our bound from Theorem~\ref{T:ComplTorusKnots} is never worse than the bound of~\cite{Casali}, and often it is strictly sharper. 
\marg{\textcolor{white}{B: In fact, the two bounds coincide only for $T(p,p-1)$ for all $p$, and for $T(7,4)$ and $T(5,3)$, I think.}}
For example, for the $(11,4)$-torus knot the two upper bounds are $4$ and $8$ respectively.

\subsection{Proof of the result on torus knots}\label{SS:TorusKnotProof}

Our aim in this section is to prove Theorem~\ref{T:ComplTorusKnots} (the bound on the complexity of the complement of the $(\alpha,\beta)$-torus knot). This knot complement is a Seifert-fibered space.  
This is well-known (see e.g.\ \cite{BurdeZieschang}), but we shall recall this construction briefly.

The $(\alpha,\beta)$-torus knot in~$S^3$ is a regular fibre of a generalized Hopf fibration of~$S^3$, as follows. Consider the mapping~$F$ from $\R^4\cong\C^2$ to the Riemann sphere
$$
F\co \C^2\to \C\cup\infty, \ (z,w)\mapsto \frac{z^\alpha}{w^\beta}
$$
This restricts to a map from the 3-sphere $S^3$ (the unit sphere in $\C^2$)
$f\co S^3\to \C\cup\infty$. Then the preimage by~$f$ of any point in the Riemann sphere is an $(\alpha,\beta)$-torus knot, with two exceptions, namely the singular fibres $f^{-1}(0)$ and $f^{-1}(\infty)$. In particular, this yields a Seifert fibration of the complement of the $(\alpha,\beta)$-torus knot $f^{-1}(1)$.

\begin{lemma}\label{L:SurgeryTorusKnot} A surgery description of the complement of the $(\alpha,\beta)$-torus knot (with $\alpha>\beta\geqslant 2$) is as follows: it is a Seifert fibered manifold of type $$(D^2, \ (\alpha,\beta'), \ (\beta,\alpha'))$$
where 
$$
0<\beta'<\alpha \text{ such that }\beta\cdot\beta'\equiv 1 \mod \alpha 
$$ 
and
$$
0<\alpha'<\beta \text{ such that }\alpha\cdot\alpha'\equiv 1 \mod \beta
$$
\end{lemma}

\begin{proof}
Let us look at the boundary of the solid torus neighbourhood of one of the singular fibres. In this torus, the intersection number of the boundary of the meridional disk with each fibre is $\beta$, and the intersection number of any longitude with each fibre is $-\alpha+k\cdot\beta$, for some integer~$k$. Thus the attaching matrix is of the form 
$$\begin{pmatrix} \beta & -\alpha+k\beta\\ \alpha' & x\end{pmatrix}$$ 
for some integers $\alpha'$ and $x$. Since the determinant of this matrix has to be one, we get $\beta x+\alpha \alpha'-k\beta\alpha'=1$, and hence 
$\alpha\cdot \alpha'\equiv 1 \ (\mathrm{mod \ }\beta)$. This means we have one surgery with parameters $(\beta,\alpha')$ where $\alpha\cdot\alpha'\equiv 1 \ (\mathrm{mod \ }\beta)$, as desired. The proof for the other surgery torus is similar.
\end{proof}

\begin{proposition}\label{P:PropertiesOfS}~
The function $S\co \N^2\to \N$ (the sum
of the terms appearing in the continued fraction expansion) has the following properties
\begin{enumerate}
\item (Symmetry) For $\alpha,\beta\in\N$ we have $S(\alpha,\beta)=S(\beta,\alpha)$.
\item For $\alpha,\beta\in\N$ with $\alpha>\beta$ we have $S(\alpha-\beta,\beta)=S(\alpha,\beta)-1$;
\item For $\alpha,\beta,\beta'\in\N$ with $\alpha>\beta\geqslant 2$ and $\alpha>\beta'\geqslant 2$ such that $\beta\cdot \beta'\equiv 1 \ {(\rm {mod} \ \alpha)}$
we have $S(\alpha,\beta')=S(\alpha,\beta)$.
\end{enumerate}
\end{proposition}

\begin{proof} The statements (1) and (2) follow directly from the definition of $S(\alpha,\beta)$. 

To prove (3) we use the Farey triangulation $\mathbb{F}$ of the hyperbolic plane $\mathbb{H}^2$ already encountered in Section~\ref{SS:theta-curves}. We recall that we define the distance between two triangles of~$\mathbb F$ to be the distance of the corresponding vertices in the dual graph to~$\mathbb F$ (which is a tree).
We also recall that we denote by $\sigma(\alpha_1/\beta_1, \alpha_2/\beta_2,
\alpha_3/\beta_3)$ the triangle of the Farey triangulation with
vertices $\alpha_1/\beta_1$, $\alpha_2/\beta_2$,
$\alpha_3/\beta_3$, and by $\sigma(\alpha/\beta)$ we denote the
closest triangle to the base triangle $\sigma(0/1, 1/0, 1/1)$
among all the triangles with the vertex $\alpha/\beta$.

\begin{lemma}\label{L:S(p,q)forTriangles}
Let $\alpha,\beta$ be coprime positive integers. Then the distance between the triangles $\sigma(0/1, 1/0, 1/1)$ and $\sigma(\alpha/\beta)$ is equal to
$S(\alpha,\beta)-1$.
\end{lemma}
\begin{proof}
We proceed by induction on $S(\alpha, \beta)$ and note that the equality is obvious when $S(\alpha, \beta)=1$, i.e.\ for $\alpha/\beta = 1/1$. Let us then consider the case $S(\alpha, \beta)>1$. Suppose first that $\alpha>\beta$.

We now recall that $GL_2(\mathbb{Z})$ acts on $\mathbb{H}^2$, using the half-plane model, by fractional linear or anti-linear transformations (depending on the sign of the determinant). All these transformations are isometries of $\mathbb{H}^2$ that leave the Farey triangulation invariant.

Let $f$ be the automorphism of $\mathbb{H}^2$ associated to {\tiny $ \begin{pmatrix} 1 & -1\\ 0 & 1\end{pmatrix}$}.
Then $f(\sigma(\alpha/\beta))=\sigma((\alpha-\beta)/\beta)$, and
$f(\sigma(0/1, 1/0, 1/1))=\sigma(-1/1, 1/0, 0/1)$, so the distance
between the triangles $\sigma(0/1, 1/0, 1/1)$ and $\sigma((\alpha-\beta)/\beta)$ is
one less than the distance between the triangles $\sigma(0/1, 1/0,
1/1)$ and $\sigma(\alpha/\beta)$. Since $S(\alpha-\beta,\beta)=S(\alpha,\beta)-1<S(\alpha,\beta)$, the
induction assumption now implies that the distance between the
triangles $\sigma(0/1, 1/0, 1/1)$ and $\sigma((\alpha-\beta)/\beta)$ is equal
to $S(\alpha-\beta,\beta)-1 = S(\alpha,\beta)-2$, whence the conclusion. If $\beta>\alpha$ we proceed similarly, using the automorphism 
{\tiny $\begin{pmatrix} 1 & 0\\ -1 & 1\end{pmatrix}$}.
\end{proof}

We now continue the proof of Proposition~\ref{P:PropertiesOfS} (3).
By the lemma and by point (1) it is now sufficient to show that the distance
between the triangles $\sigma(0/1, 1/0, 1/1)$ and $\sigma(\alpha/\beta)$ is
equal to the distance between the triangles $\sigma(0/1, 1/0,
1/1)$ and $\sigma(\beta'/\alpha)$. We also note that the desired equality
is obvious if $\beta=1$, so we proceed assuming that $\alpha/\beta$ is not an
integer.

Now let $s>0$ be such that $\beta\beta'=s\alpha+1$ and consider the following triples of vertices in the triangulation $\mathbb{F}$:
$$\left(\frac {\alpha-\beta'}{\beta-s}, \frac \alpha\beta, \frac {\beta'}s\right), \qquad \left(\frac s\beta,\frac {\beta'}\alpha,\frac {\beta'-s}{\alpha-\beta}\right).$$ It is not hard to see that the triples define $\sigma(\alpha/\beta)$ and $\sigma(\beta'/\alpha)$ respectively. Consider now the automorphism $f$ corresponding to {\tiny $\begin{pmatrix} \alpha & -\beta'\\ \beta & -s\end{pmatrix}$}.
Then $f(\sigma(0/1, 1/0, 1/1))=\sigma(\alpha/\beta)$ and $f(\sigma(\beta'/\alpha))=\sigma(0/1, 1/0, 1/1)$. Since~$f$ acts isometrically on $\mathbb{H}^2$, the distance between the triangles $\sigma(0/1, 1/0, 1/1)$ and
$\sigma(\alpha/\beta)$ is equal to the distance between the triangles
$\sigma(0/1, 1/0, 1/1)$ and $\sigma(\beta'/\alpha)$, and the proof of Proposition~\ref{P:PropertiesOfS} is complete.
\end{proof}

\begin{proof}[Proof of Theorem~\ref{T:ComplTorusKnots}]
From the surgery description of Lemma~\ref{L:SurgeryTorusKnot} and from Theorem~\ref{T:ComplSFS} we obtain the following upper bound on the complexity of the complement of the $(\alpha,\beta)$-torus knot
$$
\text{complexity}\leqslant \max\{S(\alpha,\beta')-3, 0\}+\max\{S(\beta,\alpha')-3, 0\}
$$
Now $S(\alpha,\beta')=S(\alpha,\beta)$, by Proposition~\ref{P:PropertiesOfS}(3). As far as the second term is concerned, we claim that 
$$\textstyle
S(\beta,\alpha') = S(\alpha,\beta)-\lfloor\frac{\alpha}{\beta}\rfloor
$$
In order to prove this, we define $\alpha''=\alpha-\lfloor\frac{\alpha}{\beta}\rfloor\cdot \beta$. Notice that $\alpha''\cdot \alpha'\equiv 1  \ (\mathrm{mod} \ \beta)$. Thus
$$\textstyle
S(\beta,\alpha')=S(\beta,\alpha'')=S(\beta,\alpha)-\lfloor\frac{\alpha}{\beta}\rfloor = S(\alpha,\beta)-\lfloor\frac{\alpha}{\beta}\rfloor
$$
where the first equality follows from Proposition~\ref{P:PropertiesOfS}(3), the second one from Proposition~\ref{P:PropertiesOfS}(2), and the third one from Proposition~\ref{P:PropertiesOfS}(1). This completes the proof.
\end{proof}


\end{document}